\documentclass[preprintnumbers,amsmath,11pt,amssymb,floatfix,superscriptaddress,nofootinbib]{article}
\usepackage{amsmath,amssymb,latexsym,textcomp,mathrsfs}
\usepackage[all]{xy}
\usepackage{graphicx}
\usepackage{bm,amsmath, amsthm, amssymb, amsfonts}
\usepackage{times}
\usepackage[english]{babel}
\usepackage{setspace}
\setbox0=\hbox{$+$}
\newdimen\plusheight
\plusheight=\ht0
\def\+{\;\lower\plusheight\hbox{$+$}\;}

\setbox0=\hbox{$-$}
\newdimen\minusheight
\minusheight=\ht0
\def\-{\;\lower\minusheight\hbox{$-$}\;}

\setbox0=\hbox{$\cdots$}
\newdimen\cdotsheight
\cdotsheight=\plusheight
\def\cds{\lower\cdotsheight\hbox{$\cdots$}}

\numberwithin{equation}{section}
\theoremstyle{plain}
\newtheorem{theorem}{Theorem}[section]
\newtheorem{lemma}{Lemma}[section]

\newtheorem{corollary}{Corollary}[section]

\newtheorem{definition}{Definition}[section]

\newtheorem{proposition}{Proposition}[section]

\newtheorem{note}{Note}[section]

\def\mytitle#1{\setcounter{equation}{0}
\setcounter{footnote}{0}
\begin{flushleft}\Large\textbf{#1}\end{flushleft}
\vspace{0.20cm}}
\def\myname#1{\leftline{{\large #1}}\vspace{-0.13cm}}
\def\myplace#1#2{\small\begin{flushleft}\textit{#1}\\
\texttt{#2}\end{flushleft}}

\def\myclassification#1{\small\noindent
Keywords : Weak $I^K$-convergence, weak $I^K$-Cauchy, condition AP($I,K$), weak* $I^K$-Cauchy, weak $I^K$-divergence\\
AMS subject classification: Primary: 40A05; Secondary: 40H05
       #1\vspace{0.5cm}}
\usepackage{graphicx}
\usepackage{amsmath}

\begin{document}
\mytitle{On weak $I^K$-Cauchy sequences in normed spaces}

\myname{$Amar Kumar Banerjee^{\dag}$\footnote{akbanerjee1971@gmail.com, akbanerjee@math.buruniv.ac.in} and  $Mahendranath~ Paul^{\dag}$\footnote{mahendrabktpp@gmail.com}}
\myplace{$\dag$Department of Mathematics, The University of Burdwan, 713104, India.} {}
\begin{abstract}
In this paper, we study on weak $I^K$-Cauchy condition as a generalization of weak $I^*$-Cauchy condition in a normed space. We investigate the relationship between weak $I$-Cauchy and weak $I^K$-Cauchy sequences using $AP(I,K)$-condition. Also we study on weak* $I^K$-Cauchy condition and weak $I^K$-divergence of sequences in the same space. 
 
\end{abstract}
\myclassification{}
\section{Introduction}
In the middle of twenth century, the usual idea of convergence was extended into statistical convergence by Fast \cite{17}(see also Steinhaus \cite{32}) using the concept of natural density of sets. After remarkable works of Fridy \cite{19} and salat \cite{31}, statistical convergence has become an active area of research in last fifty years. In 1985, Fridy \cite{18} introduced the concept statistically Cauchy sequences in the following way:\\
A sequence $\{x_n\}_{n\in \mathbb{N}}$ is said to be statistically Cauchy if for each $\epsilon(>0)$ there exists a number $N$ (depending on $\epsilon$) such that the set $K(\epsilon)=\{n\in \mathbb{N}:|x_n-x_N|\geq \epsilon\}$ has natural density zero, where the natural density of a subset $K$ of $\mathbb{N}$ is given as follows: Let $K_n$ denotes the set $\{k \in K:k\leq n$\} and $|K_n|$ stands for the cardinality of $K_n$.The natural density of $K$ is defined by $d(K)=\displaystyle{\lim_{n}}\frac{|K_n|}{n}$ if the limit exists. \\
In 2000, after the introductory work of weak statistical convergence by cannor et al. \cite{12}, Bhardwaj et al. \cite{1} studied on weak statistically Cauchy sequences in a normed space and proved that weak statistically Cauchy sequences are the same as weak statistically convergent sequences in a reflexive space.\\
Using the notion of the ideal $I$ of subsets of the set $\mathbb{N}$ of positive integers, the concepts of ideal convergence (namely $I$ and $I^*$-convergence) was introduced as a natural generalization of statistical convergence by Kostyrko et al.\cite{21,22}. The idea of $I$-Cauchy condition and $I^*$-Cauchy condition in metric spaces were studied by Dems \cite{13} and Nabiev et al. \cite{26} respectively. In 2010, Das et al.\cite{16} investigated some further results on $I$-Cauchyness and established an inter relation between $I$ and $I^*$-Cauchy condition using AP condition. They also introduced the concept of $I$ and $I^*$-divergence and studied on equivalency between them in \cite{26}. In 2010, Pehlivan et al. \cite{29} studied on weak $I$-convergence and weak $I$-Cauchy sequences in a normed space investigating their basic properties. Later Pal et al. \cite{27} studied on the idea of weak $I$ and $I^*$-divergence in normed linear spaces and made some important observations. The idea of weak* ideal convergence of sequences of functionals was studied by Bhardwaj et al. \cite{2}.\\
In 2010, The idea of $I^K$-convergence as a common generalization of $I^*$-convergence of sequences was given by Macaj et al. \cite{24}. In 2014, Das et al. \cite{15} studied on $I^K$-Cauchy and $I^K$-divergence and strengthen the importance of AP($I,K$) condition in the study of summability theory through double ideals. Later many works on $I$-convergence were done in metric spaces \cite{4,5,10,11}. Recently, in \cite{11.1} the idea of weak and weak* $I^K$-convergence of sequences were studied in a normed space.\\
In this paper we study the idea of weak $I^K$-Cauchy condition as a generalization of weak $I^*$-Cauchy condition. Next we show the interrelation between weak $I$-Cauchy and weak $I^K$-Cauchy using the AP$(I,K)$-condition. Also we study the idea of weak* $I^K$-Cauchy condition of sequences of functionals and investigate on weak $I^K$-divergence as a generalization of weak $I^*$-divergence.
\section{Basic definitions and notations}
\begin{definition}
Let $X$ be a non void set and a class $I\subset 2^X$ of subsets of $X$ is said to an ideal if\\
(i)$A,B\in I$ implies $A\cup B\in I$ and 
(ii)$A\in I, B\subset A$ implies $B\in I$. 
\end{definition}
It is easy to observe from condition (ii) that $\phi\in I$. $I$ is said to be nontrivial ideal if $I\neq 2^{X},\{\phi\}$. A nontrivial ideal $I$ is called admissible if it contains all the singletons of $X$. A nontrivial ideal $I$ is said to be non-admissible if it is not admissible.

If $I$ is a non-trivial ideal on a non-void set $X$ then $F=F(I)=\{A\subset X:X\setminus A \in I \}$ is clearly a filter on $X$ which is called the filter associated with the ideal $I$.
Throughout the paper, $\mathbb{N}$ denotes the set of all positive integers and $X$ denotes a normed linear space and $X^*$ is the dual of $X$.
\begin{definition}\cite{17}
A sequence $\{x_n\}$ in $X$ is said to be statistically convergent to $l$ if for every $\epsilon>0$ the set $K(\epsilon)=\{k\in\mathbb N:||x_k-l||\geq \epsilon\}$ has natural density zero.
\end{definition}
\begin{definition}\cite{12}
A sequence $\{x_n\}_{n\in\mathbb{N}}$ in $X$ is said to be weakly statistically convergent to $x\in X$ if for any $f\in X^*$, the sequence $\{f(x_n-x)\}_{n\in\mathbb{N}}$ is statistically convergent to $0$. In this case we write $w-st-\displaystyle{\lim_{n\rightarrow\infty}}{x_n}=x$.
\end{definition}
\begin{definition}\cite{1}
A sequence $\{x_n\}$ in a normed space $X$ is said to be weak statistically Cauchy
if $\{f(x_n)\}$ is statistically Cauchy for every $f\in X^*$.
\end{definition}

\begin{definition}\cite{21}
A sequence $\{x_n\}_{n\in \mathbb{N}}$ in $X$ is said to be $I$-convergent to $x$ if for any $\epsilon>0$ the set $A(\epsilon)=\{n\in\mathbb{N}:||x_n-x||\geq \epsilon\}\in I$. In this case we write $I-\displaystyle{\lim_{n\rightarrow\infty}}{x_n}=x$.
\end{definition}
\begin{definition}\cite{29}
A sequence $\{x_n\}_{n\in \mathbb{N}}$ in $X$ is said to be weakly $I$-convergent to $x\in X$ if for any $\epsilon>0$ and $f\in X^*$ the set $A(f,\epsilon)=\{n\in\mathbb{N}:|f(x_n)-f(x)|\geq \epsilon\}\in I$. In this case we write $w-I-\displaystyle{\lim_{n\rightarrow\infty}}{x_n}=x$.
\end{definition}
\begin{definition}(cf.\cite{29})
A sequence $\{x_n\}_{n\in \mathbb{N}}$ in $X$ is said to be weakly $I^*$-convergent to $x$ of $X$ if there exists a set $M\in F(I)$ such that the sequence $\{y_n\}_{n\in \mathbb{N}}\in X$ defined by
\[y_n=\left\{\begin {array}{ll}
        x_n & \mbox{if $n\in M$} \\
		x & \mbox{if $n\notin M$}
		\end{array}
		\right. \]
is weakly convergent to $x$. we denote it by the notation $w-I^*-\lim x_n=x$.
\end{definition} 
\begin{definition}
Let $I,K$ be two ideals on the set $\mathbb{N}$. A sequence $\{x_n\}_{n\in \mathbb{N}}$ in $X$ is said to be weakly $I^K$-convergent to $x$ of $X$ if there exists a set $M\in F(I)$ such that the sequence $\{y_n\}_{n\in \mathbb{N}}\in X$ defined by
\[y_n=\left\{\begin {array}{ll}
        x_n & \mbox{if $n\in M$} \\
		x & \mbox{if $n\notin M$}
		\end{array}
		\right. \]
is weakly $K$-convergent to $x$. we denote it by the notation $w-I^K-\lim x_n=x$.
\end{definition}

\begin{definition} \cite{24}
Let $K$ be an ideal $\mathbb{N}$. We denote $A\subset_K B$ whenever $A\setminus B\in K.$ If $A\subset_K B$ and $B\subset_K A$ then we denote $A\sim_K B$. Clearly $A\sim_K B \Leftrightarrow A\bigtriangleup B\in K$.\\
We say that a set $A$ is $K$-pseudo intersection of a system $\{A_n: n\in \mathbb{N}\}$ if $A\subset_K A_n$ holds for each $n\in \mathbb{N}$
\end{definition}
\begin{definition} \cite{24}\label{2}
Let $I,K$ be ideals on the set $X$. We say that $I$ has additive property with respect to $K$ or that the condition AP$(I,K)$ holds if any one of the following equivalent condition holds:
\item[(i)] For every sequence $\{A_i\}_{i\in \mathbb{N}}$ of sets from  $I$ there is $A\in I$ such that $A_i\subset_K A$ for all $i'$s.
\item[(ii)] Any sequence $\{F_i\}_{i\in \mathbb{N}}$ of sets from $F(I)$ has $K$-pseudo intersection in $F(I)$.
\item[(iii)] For every sequence $\{A_i\}_{i\in \mathbb{N}}$ of sets from  $I$ there exists a sequence$\{B_i\}_{i\in \mathbb{N}}\in I$ such that $A_n\sim_K B_n$ for $n\in \mathbb{N}$ and $B=\displaystyle{\cup_{n\in \mathbb{N}}}B_n\in I$.
\item[(iv)] For every sequence of mutually disjoint sets $\{A_i\}_{i\in \mathbb{N}}\subset I$ there exists a sequence $\{B_i\}_{i\in \mathbb{N}}\subset I$ such that $A_n\sim_K B_n$ for $n\in \mathbb{N}$ and $B=\displaystyle{\cup_{n\in \mathbb{N}}}B_n\in I$.
\item[(v)] For every non-decreasing sequence $A_1\subseteq A_2\subseteq \cdots \subseteq A_i\cdots $ of sets from $I$ $\exists$  a sequence $\{B_i\}_{i\in \mathbb{N}}\subset I$ such that $A_n\sim_K B_n$ for $n\in \mathbb{N}$ and $B=\displaystyle{\cup_{n\in \mathbb{N}}}B_n\in I$.
\item[(vi)] In the Boolean algebra $2^S/K$ the ideal $I$ corresponds to a $\sigma$-directed subset,i.e. every countable subset has an upper bound.
\end{definition}
Above definition is reformulation of the definition given below:
\begin{definition}\cite{15}
Let $I,K$ be ideals on the non-empty set $S$. We say that $I$ has additive property with respect to $K$ or that the condition AP$(I,K)$ holds if for every sequence $\{A_n\}$ of pairwise disjoint sets belonging to $I$, there exists a sequence $\{B_n\}$ belonging to $I$ such that $A_n\bigtriangleup B_n\in K$ for each $n$ and $\displaystyle \cup_{n\in \mathbb N}B_n\in I$.
\end{definition}
\section{Weak $I^K$-Cauchy}
\begin{definition}\cite{13}
A sequence $\{x_n\}_{n\in \mathbb{N}}$ is said to be $I$-Cauchy if for each $\epsilon>0$ there exists a number $N$ (depending on $\epsilon$) such that the set $\{n\in \mathbb{N}:|x_n-x_N|\geq \epsilon\}\in I$.
\end{definition}
\begin{definition}\cite{29}
A sequence $\{x_n\}$ in $X$ is said to be a weak $I$-Cauchy provided that for each $\epsilon>0$ and $f\in X^*$, there exists a possetive integer $N=N(\epsilon, f)$ such that $\{n\in \mathbb{N}:|f(x_n)-f(x_N)|\geq \epsilon\}\in I$.
\end{definition} 

\begin{note}\label{1}
It is easy to show that if $I$ and $I_1$ are any two ideals on $\mathbb{N}$ such that $I\subseteq I_1$ then for a sequence $\{x_n\}$ $w-I-\lim{x_n}=x$ implies $w-I_1-\lim{x_n}=x$.
\end{note}
\begin{lemma} (cf \cite{15})\label{3.1}
Let $X$ be a normed space and let $I$ be an ideal on $\mathbb{N}$. Then for any sequence $\{x_n\}$ the following conditions are equivalent:\\
(i) $\{x_n\}$ is weak $I$-Cauchy.\\
(ii) For any $\epsilon>0$ there is an integer $N$ such that $\{n\in \mathbb{N}:|f(x_n)-f(x_N)|<\epsilon\}\in F(I)$.\\
(iii) For every $\epsilon(>0)$ there exists a set $A\in I$ such that $m,n\notin A$ implies $|f(x_m)-f(x_n)|<\epsilon$.
\end{lemma}
\begin{lemma}
A weakly $I$-convergent sequence $\{x_n\}$ is weak $I$-Cauchy.
\end{lemma}
\begin{proof}
The proof is straightforward and so omitted. 
\end{proof}
In this paper, our main aim is to generalize the notion of weak $I^*$-Cauchyness of sequences. The definition of weak $I^*$-Cauchy sequences as introduced in\cite{26} is as follows.
\begin{definition}(\cite{26})
A sequence $\{x_n\}_{n\in \mathbb{N}}$ in $X$ is said to be weak $I^*$-Cauchy if there exists a set $M=\{m_1<m_2<...<m_k<...\}\in F(I)$ such that the subsequence $\{f(x_{m_k})\}$ is Cauchy in the ordinary sense for each $f\in X^*$. 
\end{definition} 
\begin{definition}
Let $I,K$ be two ideals on the set $\mathbb{N}$. A sequence $\{x_n\}_{n\in \mathbb{N}}$ in $X$ is said to be weak $I^K$-Cauchy if there exists a set $M\in F(I)$ such that the sequence $\{x_n\}_{n\in M}$ is weak $K|_M$-Cauchy, where $K|_M=\{A\cap M: A\in K\}$
\end{definition} 

\begin{lemma}
Every weakly $I^K$-convergent sequence is weak $I^K$-Cauchy.\\
- The proof is directly follows from the definition and so is omitted.
\end{lemma}
\begin{proposition}\label{2}
Let $X$ be a normed space. Let $I,I_1,K,K_1$ be ideals on the set $\mathbb{N}$ such that $I\subseteq I_1$ and $K\subseteq K_1$.\\
(i) If $\{x_n\}$ is weak $I^K$-Cauchy then it is also weak $I_1^K$-Cauchy.\\
(ii) If $\{x_n\}$ is weak $I^K$-Cauchy then it is also weak $I^{K_1}$-Cauchy.
\end{proposition}
\begin{proof}
If $\{x_n\}$ is weak $I^K$-Cauchy then there exists a subset $M\in F(I)$ such that $\{x_n\}_{n\in M}$ is weak $K|_M$-Cauchy.\\
(i) Now as $I\subseteq I_1$ so $F(I)\subseteq F(I_1)$. Then we have $M\in F(I_1)$ and so $\{x_n\}$ is also weak $I_1^K$-Cauchy.\\
(ii) Since $K\subseteq K_1$ so $K|_M\subseteq K_1|_M$. So from the lemma\ref{3.1}, we get that if $\{x_n\}_{n\in M}$ is weak $K|_M$-Cauchy then it is also weak $K_1|_M$-Cauchy. Hence $\{x_n\}$ is weak $I^{K_1}$-Cauchy
\end{proof}
\begin{lemma}\label{0}
If $I$ and $K$ are ideals on $\mathbb{N}$ and if a sequence $\{x_n\}_{n\in\mathbb{N}}$ is weak $K$-Cauchy, then it is also weak $I^K$-Cauchy.
\end{lemma}
\begin{proof}
If we take $M=\mathbb{N}$ then $M\in F(I)$, in this case $K|_M=K$, hence $\{x_n\}_{n\in M}$ is weak $K|_M$-Cauchy. This shows that $\{x_n\}$ is weak $I^K$-Cauchy.
\end{proof}
\begin{proposition}
 A sequence $\{x_n\}_{n\in \mathbb{N}}\in X$ is weak $I^I$-Cauchy if and only if it is weak $I$-Cauchy.
\end{proposition}
\begin{proof}
Since $\{x_n\}$ is weak $I^I$-Cauchy so there exists an $M\in F(I)$ such that the sequence $\{x_n\}_{n\in M}$ is weak $I|_M$-Cauchy i.e. for each $\epsilon(>0)$ there is $F\in F(I)$ such that $m,n\in F\cap M \Rightarrow |f(x_m)-f(x_n)|<\epsilon$. Since the set $F\cap M \in F(I)$, this shows that $\{x_n\}$ is weak $I$-Cauchy.\\
Converse part follows from the lemma \ref{0} by taking $K=I$.
\end{proof}
\begin{corollary}
Let $I,K$ be two ideals on $\mathbb{N}$. For a sequence $\{x_n\}$ in $X$, weak $I^K$-Cauchyness implies weak $I$-Cauchyness if $K\subset I$. 
\end{corollary}
\begin{proof}
If the sequence is $\{x_n\}$ weak $I^K$-Cauchy then it follows from the proposition \ref{2}(ii) that it is weak $I^I$-Cauchy. So by previous theorem we get that $\{x_n\}$ is weak $I$-Cauchy.
\end{proof}
For any two ideals $I,K$ on $\mathbb{N}$, the set $I\vee K=\{A\cup B: A \in I, B\in K \}$ is an ideal on $\mathbb{N}$ and this the smallest ideal containing both $I$ and $K$, The dual filter is $F(I \vee K)= F(I)\vee F(K)= \{F\cap G:F\in F(I), G\in F(K)\}$

\begin{theorem}
Let $I,K$ be ideals on a set $\mathbb{N}$. A sequence $\{x_n\}$ in $X$ is weak $I^K$-Cauchy if and only if it is weak $(I\vee K)^K$-Cauchy.
\end{theorem}
\begin{proof}
Since $I\subset (I\vee K)$ so by proposition \ref{2}(i) we get that weak $I^K$-Cauchyness implies weak $(I\vee K)^K$-Cauchyness of $\{x_n\}$ .\\
 Conversely, assume that $\{x_n\}$ is weak $(I\vee K)^K$-Cauchy so there exists a set $M\in F(I\vee K)$ such that $\{x_n\}|_{n\in M}$ is weak $K|_M$-Cauchy that is for every $\epsilon (>0)$ and for every $f\in X^*$ there exists $F\in F(K)$ such that $m,n \in F\cap M \Rightarrow |f(x_n)-f(x_m)|<\epsilon$. Again it is obvious that $M=M'\cap G$ for some $M'\in F(I)$ and $G\in F(K)$. If we put $F'= F\cap G$ then $F'\in F(K)$ and $F'\cap M'=F\cap G \cap M'=F\cap M$ which gives $n,m\in F'\cap M'\Rightarrow |f(x_n)-f(x_m)|<\epsilon$. So $\{x_n\}$ is weak $I^K$-Cauchy.
\end{proof}
\begin{theorem}
Let $X$ be a normed space and $I, K$ be ideals on $\mathbb{N}$ then a weak $I$-Cauchy sequence $\{x_n\}$ in $X$ is weak $I^K$-Cauchy if and only if AP$(I,K)$ condition holds.
\end{theorem}
First suppose that AP$(I,K)$ condition holds. Since $\{x_n\}$ is weak $I$-Cauchy, so for every $r\in \mathbb{N}$, we can find a set $G_r\in I$ such that $s,t\notin G_r$ implies $|f(x_s)-f(x_t)|<\frac{1}{r}$. Let $A_1=G_1$, $A_2=G_2\setminus G_1$, $A_3=G_3 \setminus(G_1\cup G_2)...$ etc. Then $\{A_i:i\in \mathbb{N}\}$ is a countable family of mutual disjoint sets in $I$. Since AP($I,K$) condition holds, there exists a family of sets $\{B_i:i\in \mathbb{N}\}\in I$ such that $B=\displaystyle{\cup_{i\in \mathbb{N}}B_i\in I}$ and for every $j\in \mathbb{N}$ we have $A_j\bigtriangleup B_j\in K$ i.e. $A_j\bigtriangleup B_j=\mathbb{N}\setminus C_j$ for some $C_j\in F(K)$. Let $M=\mathbb{N}\setminus B$ and $\epsilon>0$ be given. Choose $r\in \mathbb{N}$ such that $\frac{1}{r}<\epsilon$. Now, $G_r\cap M= G_r\setminus B\subseteq \displaystyle{\cup_{i=1}^r}(A_i\setminus B_i)\subseteq \displaystyle{\cup_{i=1}^r}(\mathbb{N}\setminus C_i)=\mathbb{N}\setminus C$, where $C=\displaystyle{\cap_{i=1}^r}C_i\in F(K)$. Then $G_r^c\cap M \cap C\supseteq M\cap C$(for otherwise there is an $s\in M\cap C$ but $s\notin G_r^c$ that is $B\in G_r$ and so $s\in G_r\cap M \cap C\subseteq \mathbb{N}\setminus C$ which is a contradiction.) This shows that $s,t\in C\cap M$ implies $s,t\in G_r^c$ which implies $|f(x_s)-f(x_t)|<\frac{1}{r}<\epsilon$. But $C\cap M\in F(K|_M)$ i.e. $\{f(x_n)\}_{n\in M}$ is $K|_M$-Cauchy i.e. $\{x_n\}_{n\in M}$ is weak $K|_M$-Cauchy. Hence the sequence $\{x_n\}_{n\in M}$ is weak $I^K$-Cauchy.
\begin{theorem}
Let $X$ be a normed space and $I, K$ be ideals on $\mathbb{N}$. Suppose that every weak $I$-Cauchy sequence is weak $I^K$-Cauchy. Then the condition AP($I,K$) holds.
\end{theorem}
\begin{proof}
Let $x_0\in X$ be a point such that $x_0$ is not an isolated point and let $\{y_n\}_{n\in \mathbb{N}}$ be a sequence of distinct points in $X$ such that $\{y_n\}$ convergent to $x_0$. Let $\{A_i:i\in \mathbb{N}\}$ be a sequence of mutually disjoint non-empty sets from $I$. We define a sequence $\{x_n\}$ by
\[x_n=\left\{\begin {array}{ll}
        y_j & \mbox{if $n\in A_j$} \\
		x_0 & \mbox{if $n\notin A_j$ for any $j$}
		\end{array}
		\right. \]
Let $f\in X^*$ be arbitrary and $\epsilon(>0)$ be given. Choose $k\in \mathbb{N}$ such that $|f(y_n)-f(x_0)|<\frac{\epsilon}{2}$ for each $n\geq k$. Now the set $D=\{n\in \mathbb{N}:|f(y_n)-f(x_0)|\geq \epsilon\}$ has the property that $D\subseteq A_1\cup A_2\cup...\cup A_k$ and so $D\in I$ and such that $s,t\notin D$ implies $|f(x_s)-f(x_t)|\leq |f(x_s)-f(x_0)|+|f(x_t)-f(x_0)|<\epsilon$. This shows that $f(x_n)$ is $I$-Cauchy i.e. $\{x_n\}$ is weak $I$-Cauchy. Then by our assumption $\{x_n\}$ is weak $I^K$-Cauchy. Hence there is a set $M\in F(I)$ such that $\{x_n\}_{n\in M}$ is weak $K|_M$-Cauchy. Let $B_j=A_j\setminus M$. Then $B_j\in I$ for each $j$ and $\displaystyle{\cup_j B_j}\subseteq \mathbb{N}\setminus M\in I$, clearly $A_j\bigtriangleup B_j=A_j\setminus B_j=A_j\cap M$ for every $j$. There may arrise two cases.\\
Case-1: If $A_j\cap M\in K$ for all $j\in \mathbb{N}$ then AP$(I,K)$ holds. If $A_j\cap M\notin K$ for at most one $j$, say $j_0$, then redefining $B_{j_0}=A_{J_0}$, $B_j=A_j\setminus M$ when $j\neq j_0$, we again observe that the condition AP($I,K$) holds.\\
Case-2: Finally if possible suppose that $A_j\cap M \notin K$ for at least two $j's$. Let $l_1$ and $l_2$ be two such indices. We shall show that this is not possible. We have $C_1=A_{l_1}\cap M\notin K$ and $C_2=A_{l_2}\cap M \notin K$ where $l_1\neq l_2$. Now any $E\in F(K|_M)$ is of the form $E=C\cap M$ where $C\in F(K)$. Now $C\cap C_1\neq \phi$ ( otherwise we will have $C_1\subseteq \mathbb{N}\setminus C\in K$, which is not the case. By similar reasoning $C\cap C_2\neq \phi$.) Then there is an $s\in E$ such that $s\in A_{l_1}$ i.e. $x_s=y_{l_1}$ and a $t\in E$ such that $t\in A_{l_2}$ i.e. $x_t=y_{l_2}$. Now choosing $0<\epsilon_0<\frac{|f(y_{l_1})-f(y_{l_2})|}{3}$ we observe that for every $E\in F(K|_M)$ there exist points $s,t\in E$ such that $|f(x_s)-f(x_t)|>\epsilon_0$ or, in other words, for this $\epsilon_0$, there does not exists any $D\in K|_M$ such that $s,t\notin D$ implies $|f(x_s)-f(x_t)|<\epsilon_0$. But this contradicts the fact that $\{x_n\}_{n\in M}$ is weak $I^K$-Cauchy. Hence case-2 can not arise.
\end{proof}
\section{Weak* $I$-Cauchy}
\begin{definition}
A sequence $\{f_n\}$ in $X^*$ is said to be weak* $I$-Cauchy sequence if for each $\epsilon(>0)$ and $x\in X$ there exists a positive integer $N$ such that $\{n\in \mathbb{N}:|f_n(x)-f_N(x)|\geq \epsilon\}\in I$.
\end{definition}
\begin{definition}
A sequence $\{f_n\}$ in $X^*$ is said to be weak* $I^*$-Cauchy sequence if there exists a set $M=\{m_1<m_2<m_3<...<m_k<...\}\in F(I)$ such that $\{f_{m_k}(x)\}$ is Cauchy in ordinary sense for every $x\in X$.
\end{definition}
\begin{theorem}
Let $X$ be a normed space with dual $X^*$ and $\{f_n\}$ be a sequence in $X^*$. If $\{f_n\}$ is weak* $I^*$-Cauchy then it is weak* $I$-Cauchy.
\end{theorem}
\begin{proof}
Let $\{f_n\}$ be a weak* $I^*$-Cauchy sequence in $X^*$. Now from the definition we have there exists a set $M=\{m_1<m_2<m_3<...<m_k<...\}\in F(I)$ such that $\{f_{m_k}(x)\}$ is Cauchy in ordinary sense for every $x\in X$, that is, for every $\epsilon(>0)$ there exists a positive integer $k_0$ such that $|f_{m_k}(x)-f_{m_p}(x)|<\epsilon$  for all $k,p>k_0$. Let $N=N(\epsilon,x)=m_{k_0+1}$ then we have $|f_{m_k}(x)-f_N(x)|<\epsilon$ when $k>k_0$. Now $M_0=\mathbb{N}\setminus M\in I$ and $A(\epsilon,x)=\{n\in \mathbb{N}:|f_n(x)-f_N(x)|\geq \epsilon\}\subset M_0\cup \{m_1,m_2,...,m_{k_0}\}\in I$. Hence the result follows.
\end{proof}
\begin{theorem}\label{d}
Let $X$ be a norm space with dual $X^*$. If a sequence $\{f_k\}$ in $X^*$ is weak $I$-Cauchy then it is weak* $I$-Cauchy.
\end{theorem}
\begin{proof}
Let $\{f_k\}$ be a weak $I$-Cauchy in $X^*$ .Then for each $\epsilon(>0)$ and $g\in X^{**}$ there exists a positive integer $N$ such that $\{n\in \mathbb {N}: |g(f_n)-g(f_N)|\geq \epsilon\}\in I \rightarrow (1)$. Let $x\in X$ and $g_x=C(x)$ where $C:X\rightarrow X^{**}$ is the cannonical mapping. So we have, $g_x(f_n)=f_n(x)$ and $g_x(f_N)=f_N(x)$  and therefore , by (1), $\{n\in \mathbb {N}:|g_x(f_n)-g_x(f_N)|\geq \epsilon\}\in I$. Hence $\{f_k\}$ is weak* $I$-Cauchy.
\end{proof}
\begin{theorem}\label{dd}
Let $X$ be a reflexive Banach space. If a sequence $\{f_k\}$ in $X^*$ is weak* $I$-Cauchy then it is weak $I$-Cauchy.
\end{theorem}
\begin{proof}
Let $X$ be a reflexive space. Since $\{f_k\}$ is weak* $I$-Cauchy then for each $x\in X$ and $\epsilon(>0)$ there exists positive integer $N$ such that  the set $\{n\in \mathbb {N}:|f_n(x)-f_N(x)|\geq \epsilon\}\in I$. Let $g\in X^{**}$. Since $X$ is reflexive, $g=C(x)$ for some $x\in X$ where $C:X\rightarrow X^{**}$ is a cannonical mapping. Since $g(f_n)=f_n(x)$ and $g(f_N)=f_N(x)$ we have $\{n\in \mathbb{N}:|g(f_n)-g(f_N)|\geq \epsilon\}\in I$ for each $\epsilon>0$. Since the result is true for every $g\in X^{**}$, $\{f_k\}$ is weak $I$-Cauchy.
\end{proof}
\begin{definition}
Let $I,K$ be two ideals on the set $\mathbb {N}$. A sequence $\{f_n\}$ in $X^{*}$ is said to be weak* $I^K$-Cauchy if there exists a set $M\in F(I)$ such that the sequence $\{f_n\}|_M$ is weak* $K|_M$-Cauchy where $K|_M=\{A\cap M: A\in K\}$
\end{definition}
\begin{theorem}
Let $X$ be a normed space. If a sequence $\{f_k\}$ in $X^*$ is weak $I^K$-Cauchy then it is weak* $I^K$-Cauchy.
\end{theorem}
Proof follows directly from the definition and Theorem \ref{d}.
\begin{note}
It is easy to observe that using definition of weak* $I^K$-convergence and theorem \ref{dd}, converse of the above theorem also holds when $X$ is a reflexive Banach space.
\end{note}
\section{Weak $I^K$-divergence}
\begin{definition}\cite{15}
A sequence $\{x_n\}_{n\in \mathbb{N}}$ in a normed space $X$ is said to be $I$-divergent if there exists an element $x\in X$ such that for any positive real number $G$,\\
$A(x,G)=\{n\in \mathbb{N}:|x_n-x|\leq G\}\in I$.
\end{definition}
\begin{definition}\cite{15}
A sequence $\{x_n\}_{n\in \mathbb{N}}$ in a normed space $X$ is said to be weak $I$-divergent if for each $f\in X^*$, $\{f(x_n)\}_{n\in \mathbb{N}}$ is $I$-divergent i.e. there exists an element $x\in X$ such that for any positive real number $G$,\\
$A(f(x),G)=\{n\in \mathbb{N}:|f(x_n)-f(x)|\leq G\}\in I$.
\end{definition}
\begin{definition}\cite{15}
A sequence $\{x_n\}_{n\in \mathbb{N}}$ in a normed linear space $X$ is said to be $I^*$-divergent if there exists a set $M\in F(I)$ such that $\{x_n\}_{n\in M}$ is properly divergent i.e. there exists an $x\in X$ such that $|x_n-x|\rightarrow 0$ as $n\rightarrow \infty$.
\end{definition}
\begin{definition}
A sequence $\{x_n\}_{n\in \mathbb{N}}$ in a normed space $X$ is said to be weak $I^*$-divergent if for each $f\in X^*$ there exists a set $M\in F(I)$ $\{f(x_n)\}_{n\in M}$ is properly divergent.
\end{definition}
\begin{definition}
Let $K$ and $I$ be two ideals on $\mathbb{N}$. A sequence $\{x_n\}_{n\in \mathbb{N}}$ in a normed space $X$ is said to be $I^K$-divergent if there exists a set $M\in F(I)$ such that $\{x_n\}_{n\in M}$ is $K|_M$-divergent. 
\end{definition}
\begin{definition}
Let $K$ and $I$ be two ideals on $\mathbb{N}$. A sequence $\{x_n\}_{n\in \mathbb{N}}$ in a normed linear space $X$ is said to be weak $I^K$-divergent if for each $f\in X^*$ there exists a set $M\in F(I)$ such that $\{f(x_n)\}_{n\in \mathbb{M}}$ is $K|_M$-divergent.
\end{definition}
\begin{lemma}
If a sequence $\{x_n\}$ is weak $K$-divergent then it is also weak $I^K$-divergent.
\end{lemma}
\begin{proof}
If we consider $M=\mathbb{N}\in F(I)$ and $K|_M=K$ then the proof follows from definition directly.
\end{proof}
\begin{lemma}
If a sequence $\{x_n\}$ is weak $K$-divergent then it is also weak $I^K$-divergent.
\end{lemma}
\section{Acknowledgment}
The second author is grateful to the Department of Science and Technology, Govt. of India for providing fund on FIST project to the Department of Mathematics, University of Burdwan, W.B.


\end{document}